\documentclass[reqno]{amsart}
\usepackage{amssymb,enumerate,mathrsfs,curves,amscd}

\numberwithin{equation}{section}

\newtheorem{theorem}[equation]{Theorem}

\newtheorem{lemma}[equation]{Lemma}
\newtheorem{corollary}[equation]{Corollary}

\theoremstyle{definition}

\newtheorem{example}[equation]{Example}

\def\C{\mathbb C}
\def\Dom{\mathcal{D}}

\def\K{\mathcal K}

\def\L{\mathscr L}
\def\N{\mathbb N}
\def\R{\mathbb R}

\def\cev{\hspace*{0.2ex}{}^c\hspace*{-0.2ex} {\mathrm{ev}}}
\def\cg{\hspace*{0.2ex}{}^c\hspace*{-0.2ex} g}
\def\cpi{\hspace*{0.1ex}{}^c\hspace*{-0.1ex} \pi}
\def\csym{\,{}^c\!\sym}

\def\cT{\,{}^c T}

\def\eps{\varepsilon}
\def\Mbar{\overline{M}}
\def\m{\mathfrak m}
\def\minus{\backslash}

\def\open#1{\smash[t]{\overset{{}_{\circ}}{#1}{}}}
\def\reg{\textup{reg}}

\def\sing{\textup{sing}}

\def\Vee{{\mathcal V}}

\DeclareMathOperator{\Ind}{ind}

\DeclareMathOperator{\Diff}{Diff}
\DeclareMathOperator{\spec}{spec}
\DeclareMathOperator{\sym}{ \sigma\!\!\!\sigma}
\DeclareMathOperator{\Span}{span}
\DeclareMathOperator{\Eig}{{\mathfrak E}{\mathfrak i}{\mathfrak g}}

\begin{document}
\title[Completeness of generalized eigenfunctions]{On the completeness of
generalized eigenfunctions of elliptic cone operators} 
\thanks{Partially supported by the National Science Foundation
under grant DMS 0901202.}
\author{Thomas Krainer}
\address{Penn State Altoona\\ 3000 Ivyside Park \\ Altoona, PA 16601-3760}
\email{krainer@psu.edu}
\begin{abstract}
We show the completeness of the system of generalized eigenfunctions of
closed extensions of elliptic cone operators under suitable conditions
on the symbols.
\end{abstract}

\subjclass[2000]{Primary: 58J50; Secondary: 35P10, 58J05}
\keywords{Manifolds with conical singularities, resolvents, spectral theory}

\maketitle

%%%%%%%%%%%%%%%%%%%%%%%%%%%%%%%%%%%%%%%%%%%%%%%%%%%%%%%%%%%%%%%%%%%%%
\section{Introduction}
%%%%%%%%%%%%%%%%%%%%%%%%%%%%%%%%%%%%%%%%%%%%%%%%%%%%%%%%%%%%%%%%%%%%%

\noindent
The purpose of this note is to extend the theorem about completeness of the system
of generalized eigenfunctions of elliptic operators on manifolds with conical
singularities of Egorov, Kondratiev, and Schulze
\cite{EgorovKondratievSchulze1,EgorovKondratievSchulze2} to the general case.
While it is implicit in their presentation, it is, however, important to note
that their result is applicable only for the minimal extension of the operator.
This leaves out many important cases, including (nonselfadjoint)
realizations of Laplacians.
We will present two simple examples in Section~\ref{MainTheoremExamples} which illustrate
the relevancy of this observation.

Like Egorov, Kondratiev, and Schulze, we will follow Agmon's
approach \cite{Agmon} towards proving this result. 
This approach is based on a purely functional analytic theorem of Dunford and
Schwartz \cite[Chapter XI.9 and XI.6]{DunfordSchwartz}, see
Section~\ref{FunctAnalResult}, which reduces the task of proving completeness
of generalized eigenfunctions to showing that the embedding of the domain
of the operator into the Hilbert space is of Schatten class, and to showing
that the operator admits sufficiently many rays of minimal growth.
Agranovich uses the same approach in \cite[Section 6.4]{Agranovich1} and
\cite[Section 9.3]{Agranovich2} to address the completeness problem
for elliptic operators on smooth manifolds.

Rays of minimal growth for elliptic cone operators equipped with general domains
have been the subject of our earlier work \cite{GKM1,GKM2,GKM3} in collaboration
with J.~Gil and G.~Mendoza (for the boundaryless case), and \cite{KrainBVPCone}
(for the case of realizations subject to boundary conditions).
With these results at hand, it remains to prove that the embedding of the domain
of the operator into the Hilbert space is of Schatten class. To do this, we will
employ recent results of Buzano and Toft \cite{BuzanoToft,Toft} as well as the
explicit descriptions of domains of elliptic cone operators from \cite{Le97,GiMe01}
(in the boundaryless case) and \cite{CorSchroSei,KrainBVPCone} (for realizations
subject to boundary conditions).

The focus of Agmon's original paper \cite{Agmon} are elliptic boundary value
problems on smooth manifolds. He uses Fourier series to show that the embeddings
of the Sobolev spaces are of Schatten class. This argument was adapted by Egorov,
Kondratiev, and Schulze in \cite{EgorovKondratievSchulze1,EgorovKondratievSchulze2}.
In \cite{Agranovich1,Agranovich2}, Agranovich uses a different elegant argument to show
that the embedding of the domain of an operator $A$ is of Schatten class; this argument
is based on the Weyl asymptotics of the eigenvalues of the operator
$(A-\lambda_0)(A-\lambda_0)^*$ for a suitable $\lambda_0$ in the resolvent set of $A$.
That semibounded elliptic cone operators in the boundaryless case exhibit
Weyl asymptotics has been shown by Lesch \cite{Le97}. In the boundaryless case,
we could therefore follow Agranovich's argument to obtain what is needed to
prove completeness of the generalized eigenfunctions.
However, our approach gives a more general embedding result which readily applies
to all kinds of realizations of elliptic cone operators on conic manifolds
with or without boundary.

\medskip

\noindent
The structure of the paper is as follows:

In Section~\ref{FunctAnalResult} we review the functional analytic background
and the result of Dunford and Schwartz \cite[Chapter XI.9 and XI.6]{DunfordSchwartz}.

Section~\ref{ConeSobolevSpacesEmbeddings} is devoted to weighted cone Sobolev
spaces \cite{RBM2,SchulzeNH} and the embedding result that we need.

Section~\ref{ConeOperators} summarizes basics about elliptic cone operators, and
we review the results about rays of minimal growth from
\cite{GKM1,GKM2,GKM3,KrainBVPCone}.

We conclude this work in Section~\ref{MainTheoremExamples} with the main
theorems about the completeness of generalized eigenfunctions for general realizations
of elliptic cone operators, and the discussion of two simple examples to illustrate the
results.

\medskip

\noindent
The case of cone operators represents the simplest situation of elliptic operators
on incomplete Riemannian manifolds with corners. From this perspective, this work is
the first step towards addressing similar questions for this more general case.
The observations made in the present work will impact such future investigations.
As the examples in Section~\ref{MainTheoremExamples} show, it cannot be expected
that the scales of weighted Sobolev spaces that are considered
in the existing literature on elliptic operators on incomplete manifolds with
corners will be immediately related to the functional analytic domains of an
elliptic operator. Much work still needs to be done to describe these domains.
On the other hand, the present paper underscores that the proof of the
completeness of generalized eigenfunctions is based only on very few principles.

\medskip

\noindent
I would like to thank Juan Gil for several interesting discussions.

%%%%%%%%%%%%%%%%%%%%%%%%%%%%%%%%%%%%%%%%%%%%%%%%%%%%%%%%%%%%%%%%%%%%%
\section{Functional analytic background}\label{FunctAnalResult}
%%%%%%%%%%%%%%%%%%%%%%%%%%%%%%%%%%%%%%%%%%%%%%%%%%%%%%%%%%%%%%%%%%%%%

\noindent
Let $H$ be a complex Hilbert space, and let
$$
A : \Dom \subset H \to H
$$
be a closed, densely defined operator acting in $H$. The domain $\Dom$ is equipped with
the graph norm. Having realizations of elliptic operators in mind, we usually
write $A_{\Dom}$ to emphasize that $A$ acts in $H$ with domain $\Dom$.

Recall that a vector $0 \neq u \in H$ is called a \emph{generalized eigenvector}
of $A_{\Dom}$ associated with the eigenvalue $\lambda_0 \in \C$ if
$(A_{\Dom} - \lambda_0)^ku = 0$ for some $k \geq 1$. This entails of course that
$u$ is in the domain of the $k$-th power of $A_{\Dom}$.
Let $\Eig(A_{\Dom})$ denote the linear span of all generalized eigenvectors of
$A_{\Dom}$.

The statement that the system of generalized eigenvectors is complete
in $H$ means that $\Eig(A_{\Dom})$ is dense in $H$.

\begin{theorem}[{\cite[Corollary XI.9.31]{DunfordSchwartz}}]\label{AbstractResult}
Suppose the embedding $\Dom \hookrightarrow H$ belongs to the Schatten
class ${\mathfrak S}_p$ for some $0 < p < \infty$.
Moreover, let there be rays
$$
\Gamma_j = \{re^{i\theta_j};\; r \geq 0\}, \quad j=1,\ldots,J,
$$
in the complex plane that are rays of minimal growth for the operator $A_{\Dom}$,
and such that all angles enclosed by any two adjacent rays are $\leq \pi/p$.

Then the system of generalized eigenvectors of $A_{\Dom}$ is complete in $H$.
\end{theorem}

Recall that a ray $\Gamma = \{re^{i\theta};\; r \geq 0\} \subset \C$ is called
a \emph{ray of minimal growth} or a \emph{ray of maximal decay} for $A$ if
$$
A - \lambda : \Dom \to H
$$
is invertible for $\lambda \in \Gamma$ with $|\lambda| > 0$ sufficiently large, and
if the resolvent satisfies the estimate
$$
\|(A_{\Dom}-\lambda)^{-1}\|_{\L(H)} = O\bigl(|\lambda|^{-1}\bigr)
$$
as $|\lambda| \to \infty$ in $\Gamma$.

Moreover, for Hilbert spaces $E$ and $F$, ${\mathfrak S}_p$ is the space of all
$T \in \L(E,F)$ such that $\sum_{j=0}^{\infty}\alpha_j(T)^p < \infty$, where
$$
\alpha_j(T) = \inf\{\|T-G\|_{\L(E,F)};\; G \in \L(E,F),\; \dim R(G) \leq j\}
$$
is the $j$-th approximation number of $T$.

\medskip

\noindent
Observe that if $\Gamma$ is a ray of minimal growth, then there is a sector 
$\Lambda$ with $\open\Lambda \neq \emptyset$ and central axis $\Gamma$ such that all
rays in $\Lambda$ are rays of minimal growth for $A_{\Dom}$. This implies that
in Theorem~\ref{AbstractResult} above we can weaken the assumption to only
require that the embedding $\Dom \hookrightarrow H$ belongs to ${\mathfrak S}^+_p$,
where
\begin{equation}\label{Spplus}
{\mathfrak S}^+_p = \bigcap\limits_{q > p} {\mathfrak S}_q.
\end{equation}
Note that ${\mathfrak S}_i \subset {\mathfrak S}_j$ for $0 < i < j < \infty$.
This observation is rather useful when dealing with elliptic operators since
the embeddings of domains typically belong to ${\mathfrak S}^+_p$, where $p > 0$
depends on the order of the operator and the dimension of the underlying space
(see below).

\medskip

If the operator $A_{\Dom}$ has nonempty resolvent set $\varrho(A_{\Dom})$
as is assumed in Theorem~\ref{AbstractResult}, the condition that
the embedding $\Dom \hookrightarrow H$ belongs to the Schatten
class ${\mathfrak S}_p$ for some $0 < p < \infty$ is equivalent to
requiring that the resolvent $T = (A_{\Dom}-\lambda_0)^{-1} : H \to H$ belongs
to ${\mathfrak S}_p$ for some $\lambda_0 \in \varrho(A_{\Dom})$.
Recall that this means that the nonzero eigenvalues (counting multiplicities)
$\lambda_0(\sqrt{T^*T}) \geq \lambda_1(\sqrt{T^*T}) \geq \ldots > 0$ of
$\sqrt{T^*T}$ are $p$-summable, i.e.,
$\sum_{j=0}^{\infty}\lambda_j(\sqrt{T^*T})^p < \infty$.

In view of the identity $\sqrt{T^*T} =
\bigl[(A_{\Dom}-\lambda_0)(A_{\Dom}-\lambda_0)^*\bigr]^{-1/2}$
and the spectral theorem for selfadjoint operators, we conclude that if
$A_{\Dom}$ has compact resolvent and the eigenvalues
$0 < \mu_0 \leq \mu_1 \leq \ldots$ of $(A-\lambda_0)(A-\lambda_0)^*$ (counting
multiplicities) obey Weyl's law $\mu_j \sim \textup{Const}\cdot j^{\frac{2m}{n}}$
as $j \to \infty$, then $T$ belongs to ${\mathfrak S}^+_{n/m}$.
Here $m,n > 0$, and in applications to elliptic operators $m$ is the order of $A$
and $n$ the dimension of the underlying space. This is Agranovich's argument
from \cite{Agranovich1,Agranovich2} to prove that the embeddings of domains of
elliptic operators on smooth compact manifolds are of Schatten class.
As already mentioned in the introduction, we will follow in this paper a
different approach for realizations of elliptic cone operators.

%%%%%%%%%%%%%%%%%%%%%%%%%%%%%%%%%%%%%%%%%%%%%%%%%%%%%%%%%%%%%%%%%%%%%
\section{Embeddings of weighted cone Sobolev spaces}\label{ConeSobolevSpacesEmbeddings}
%%%%%%%%%%%%%%%%%%%%%%%%%%%%%%%%%%%%%%%%%%%%%%%%%%%%%%%%%%%%%%%%%%%%%

\noindent
We begin with a brief review of the definition of the scale of weighted $b$-Sobolev
spaces. More details can be found in \cite{RBM2,SchulzeNH}.

Let $\Mbar$ be a smooth, compact $n$-manifold with boundary $\partial\Mbar$,
and let $x \in C^{\infty}(\Mbar)$ be a defining function for $\partial\Mbar$.
Recall that this means that $x \geq 0$ on $\Mbar$, $\partial\Mbar = \{x = 0\}$,
and $dx \neq 0$ on $\partial\Mbar$.
By $L^2_b(\Mbar)$ we denote the $L^2$-space with respect to any $b$-density $\m$
on $\Mbar$. Recall that $\m$ is a $b$-density if $x\m$ is a smooth, everywhere
positive density on $\Mbar$.
The $b$-Sobolev space of smoothness $s \in \N_0$ is defined as
$$
H^s_b(\Mbar) = \{u \in {\mathcal D}'(\Mbar);\; Pu \in L^2_b(\Mbar) \textup{ for all }
P \in \Diff^m_b(\Mbar),\; m \leq s\}.
$$
Recall that $\Diff^m_b(\Mbar)$ is the space of $b$-differential operators of order
$m$, i.e., the operators of order $m$ in the enveloping algebra of differential
operators generated by $C^{\infty}(\Mbar)$ and the Lie algebra $\Vee_b$ of smooth
vector fields on $\Mbar$ that are tangential to the boundary.
For general $s \in \R$ the space $H^s_b(\Mbar)$ is defined by interpolation and
duality.
More generally, if $E$ is a (Hermitian) vector bundle on $\Mbar$, let
$x^{\gamma}H^s_b(\Mbar;E)$ be the weighted $b$-Sobolev space of sections of $E$
of regularity $s \in \R$.

Our first goal in this section is the following theorem.

\begin{theorem}\label{EmbeddingbCone}
The embedding
$$
x^{\gamma}H^s_b(\Mbar;E) \hookrightarrow x^{\gamma'}H^{s'}_b(\Mbar;E)
$$
belongs to the Schatten class ${\mathfrak S}_p$, $0 < p < \infty$, for any
$\gamma > \gamma'$ and $s > s' + n/p$.
\end{theorem}

The proof of Theorem~\ref{EmbeddingbCone} makes use of a corresponding result
about embeddings of weighted Sobolev spaces on $\R^n$. More precisely, for
$s,\delta \in \R$ let $H^{s,\delta}(\R^n) = \langle x \rangle^{-\delta}H^s(\R^n)$
(unlike in other contexts in this paper, $x$ represents the variable in $\R^n$ here).
We have the following lemma.

\begin{lemma}\label{EmbedSpRn}
The embedding
$$
\iota: H^{s,\delta}(\R^n) \hookrightarrow H^{s',\delta'}(\R^n)
$$
belongs to the Schatten class ${\mathfrak S}_p$, $0 < p < \infty$, for any
$\delta > \delta' + n/p$ and $s > s' + n/p$.
\end{lemma}
\begin{proof}
For the proof we may without loss of generality assume that $p > 1$: Otherwise,
let $N \in \N$ with $\frac{1}{N} < p$, and consider the composition of embeddings
$$
H^{s_0,\delta_0}(\R^n) \hookrightarrow H^{s_1,\delta_1}(\R^n) \hookrightarrow \ldots
\hookrightarrow H^{s_N,\delta_N}(\R^n),
$$
where $s_j = s - j\cdot\frac{s-s'}{N}$,
$\delta_j = \delta - j\cdot\frac{\delta-\delta'}{N}$,
$j=0,\ldots,N$. In view of $s_{j-1} - s_j = \frac{s-s'}{N} > \frac{n}{Np}$ and
$\delta_{j-1} - \delta_j = \frac{\delta-\delta'}{N} > \frac{n}{Np}$ and $Np > 1$
we may conclude that the embedding $H^{s_{j-1},\delta_{j-1}}(\R^n) \hookrightarrow
H^{s_j,\delta_j}(\R^n)$ belongs to ${\mathfrak S}_{Np}$ (if we take the result of
the lemma for granted for class indices greater than one).
The composition of $N$ mappings of class ${\mathfrak S}_{Np}$ belongs to
${\mathfrak S}_p$ by the general properties of these classes.

Hence assume in the sequel that $p > 1$. For $\mu,\varrho \in \R$ let
$\Lambda^{\mu,\varrho} = \langle x \rangle^{\varrho}\langle D_x \rangle^{\mu}$ and
$\tilde{\Lambda}^{\mu,\varrho} = \langle D_x \rangle^{\mu}\langle x \rangle^{\varrho}$.
Then
$$
\Lambda^{\mu,\varrho}, \tilde{\Lambda}^{\mu,\varrho} :
H^{s,\delta}(\R^n) \to H^{s-\mu,\delta-\varrho}(\R^n)
$$
are isomorphisms for all $s,\delta \in \R$, and obviously
$\bigl(\Lambda^{\mu,\varrho}\bigr)^{-1} = \tilde{\Lambda}^{-\mu,-\varrho}$.
In view of the commutative diagram
$$
\begin{CD}
H^{s,\delta}(\R^n) @>\iota>> H^{s',\delta'}(\R^n) \\
@V\Lambda^{s',\delta'}VV @AA\tilde{\Lambda}^{-s',-\delta'}A \\
H^{s-s',\delta-\delta'}(\R^n) @>\iota>> L^2(\R^n)
\end{CD}
$$
and the operator ideal property of the Schatten classes ${\mathfrak S}_p$ we may
assume without loss of generality that $s'=\delta'=0$.
Using again the operator ideal property and the commutative diagram
$$
\begin{CD}
H^{s,\delta}(\R^n) @>\iota>> L^2(\R^n) \\
@V\tilde{\Lambda}^{s,\delta}VV @| \\
L^2(\R^n) @>\Lambda^{-s,-\delta}>> L^2(\R^n)
\end{CD}
$$
we see that it suffices to show that the operator $\Lambda^{-s,-\delta} : L^2(\R^n) \to
L^2(\R^n)$ belongs to ${\mathfrak S}_p$ for $s > n/p$ and $\delta > n/p$.
This, however, is a direct consequence of \cite[Proposition 4.2]{BuzanoToft}, it
also follows from \cite{Toft} (these papers consider Schatten classes with
indices $p \geq 1$, this is why we made that reduction at the beginning of this proof).
The point here is that the symbol
$\langle x \rangle^{-\delta}\langle \xi \rangle^{-s}$ of the operator
$\Lambda^{-s,-\delta}$ belongs to $L^p(\R^{2n})$ precisely if $s > n/p$ and
$\delta > n/p$. The papers \cite{BuzanoToft,Toft} are concerned with characterizing
the Schatten class property with index $p$ for certain classes of pseudodifferential 
operators acting on $L^2(\R^n)$ in terms of $L^p$-bounds on their symbols (or on the
weight functions of the symbol classes). In our situation at hand these results are
applicable and lead to the desired conclusion.
\end{proof}

\begin{proof}[Proof of Theorem~\ref{EmbeddingbCone}]
Since the proof is based on a localization argument, we may assume without
loss of generality that $E \to \Mbar$ is the trivial line bundle. Moreover,
in view of the commutative diagram
$$
\begin{CD}
x^{\gamma}H^s_b(\Mbar) @>\iota>> x^{\gamma'}H^{s'}_b(\Mbar) \\
@Vx^{-\gamma'}VV @AAx^{\gamma'}A \\
x^{\gamma-\gamma'}H^s_b(\Mbar) @>\iota>> H^{s'}_b(\Mbar)
\end{CD}
$$
we may assume that $\gamma' = 0$.

Choose a collar neighborhood $\chi_0 : U_0 \cong [0,\eps)\times \partial\Mbar$
of the boundary. Without loss of generality, we may assume that the defining
function $x$ coincides in $U_0$ with the projection map to the coordinate in $[0,\eps)$.
Away from the boundary choose a finite collection $U_1,\ldots,U_N$
of open subsets of $\Mbar$ that are via charts $\chi_j : U_j \to \Omega_j$,
$j = 1,\ldots,N$, diffeomorphic to open bounded subsets $\Omega_j \subset \R^n$
such that $\Mbar = \bigcup_{j=0}^N U_j$.
Let $\varphi_j$, $j=0,\ldots,N$, be a smooth partition of unity subordinate to this
covering, and choose $\psi_j \in C^{\infty}(\Mbar)$ with compact support contained in
$U_j$ such that $\psi_j \equiv 1$ in a neighborhood of the support of $\varphi_j$.

With this data we further proceed to define maps $T_j : x^{\gamma}H^s_b(\Mbar) \to
H^{s'}_b(\Mbar)$, $j = 0,\ldots,N$, that belong to the Schatten class
${\mathfrak S}_p$ (provided that $s > s' + n/p$ as is assumed here), and such that
$\iota = \sum\limits_{j=0}^N T_j$.

More precisely, for $j = 1,\ldots,N$ define $T_j$ to be the composition of the maps
$$
T_j = \bigl(\chi_j^{\ast}\circ\Psi_j\bigr)\circ\iota\circ
\bigl(\chi_{j,\ast}\circ\varphi_j\bigr).
$$
Here $\chi_{j,\ast}\circ\varphi_j : x^{\gamma}H^s_b(\Mbar) \to H^{s,\delta}(\R^n)$
is the multiplication operator by $\varphi_j$ followed by push-forward with
respect to $\chi_j$, where $\delta > n/p$ can be chosen arbitrarily.
$\iota : H^{s,\delta}(\R^n) \to H^{s',0}(\R^n)$ is the embedding that belongs to
${\mathfrak S}_p$ by Lemma~\ref{EmbedSpRn}, and
$\chi_j^{\ast}\circ\Psi_j : H^{s',0}(\R^n) \to H^{s'}_b(\Mbar)$ is the multiplication
operator by $\Psi_j = \chi_{j,\ast}\psi$ followed by pull-back with respect to $\chi_j$.
All maps involved are continuous, and by the operator ideal property of ${\mathfrak S}_p$
we obtain that $T_j$ belongs to ${\mathfrak S}_p$ for $j = 1,\ldots,N$.
Observe that $T_ju = \varphi_ju$ for $j = 1,\ldots,N$.

Analogously to the other $T_j$, the operator $T_0$ is just $T_0u = \varphi_0 u$.
In order to see that it belongs to ${\mathfrak S}_p$ we proceed as follows:
Choose coordinate neighborhoods $U_{0j} \subset \partial\Mbar$, $j=1,\ldots,M$,
and charts $\chi_{0j} : U_{0j} \to \Omega_{0j}$, where $\Omega_{0j} \subset \R^{n-1}$
is open and bounded, such that $\partial\Mbar = \bigcup_{j=1}^MU_{0j}$. Choose
a smooth subordinate partition of unity $\varphi_{0j}$, $j = 1,\ldots,M$, and functions
$\psi_{0j} \in C^{\infty}(\partial\Mbar)$ with compact support in $U_{0j}$ such
that $\psi_{0j} \equiv 1$ in a neighborhood of the support of $\varphi_{0j}$.
Let $t$ be the diffeomorphism $\overline{\R}_+ \to \R$ defined by $t(x) = -\log(x)$.
We write $T_0 = \sum_{j=1}^M T_{0j}$, where each operator $T_{0j}$ is defined by
\begin{equation}\label{T0jEquation}
T_{0j} = \bigl[\bigl(\chi^{\ast}_{0}\circ\Psi_0 \bigr) \circ\bigl((t,\chi_{0j})^{\ast} \circ\Psi_{0j}\bigr)\bigr]\circ\iota\circ\bigl[\bigl((t,\chi_{0j})_{\ast} \circ\varphi_{0j}\bigr) \circ \bigl(\chi_{0,\ast}\circ\varphi_0 \bigr)\bigr].
\end{equation}
Here
$$
\bigl[\bigl((t,\chi_{0j})_{\ast} \circ\varphi_{0j}\bigr) \circ 
\bigl(\chi_{0,\ast}\circ\varphi_0 \bigr)\bigr] :
x^{\gamma}H^s_b(\Mbar) \to H^{s,\delta}(\R^n)
$$
is continuous, where $\delta > n/p$ can be chosen arbitrarily. Indeed, multiplication by
$\varphi_0$ and push-forward by $\chi_0$ localizes distributions near the boundary
and introduces the splitting of variables $(x,y) \in [0,\eps)\times\partial\Mbar$,
multiplication by $\varphi_{0j}$ localizes the $y$-dependence further to the
coordinate neighborhood $U_{0j}$, push-forward by $\chi_{0j}$ in the $y$-variable
and by $t$ in the $x$-variable produces distributions on $\R\times\R^{n-1}$
that are supported in the strip $\R\times\Omega_{0j}$ and that vanish in a
neighborhood of $t = -\infty$. The weight $x^{\gamma}$ translates into an exponential
weight $e^{-t\gamma}$ near $t = \infty$. In view of the support properties
just discussed, we see that we certainly obtain a Sobolev distribution on $\R^n$
that exhibits any polynomial decay (in the Sobolev norm), or, in other words, we
arrive in $H^{s,\delta}(\R^n)$ for any choice of $\delta > n/p$ as was claimed.

The other parts in \eqref{T0jEquation} are the embedding $\iota : H^{s,\delta}(\R^n)
\to H^{s',0}(\R^n)$ that belongs to ${\mathfrak S}_p$ by Lemma~\ref{EmbedSpRn},
and the operator
$$
\bigl[\bigl(\chi^{\ast}_{0}\circ\Psi_0 \bigr) \circ\bigl((t,\chi_{0j})^{\ast} 
\circ\Psi_{0j}\bigr)\bigr] :
H^{s',0}(\R^n) \to H^{s'}_b(\Mbar)
$$
consisting of multiplication by $\Psi_{0j} = \chi_{0j,\ast}\psi_{0j}$, pull-back
via $t$ and $\chi_{0j}$ and multiplication by $\Psi_0 = \chi_{0,\ast}\psi_0$
to yield disbritutions on $[0,\eps)\times\partial\Mbar$, and finally pull-back by
$\chi_0$ to yield distributions in $H^{s'}_b(\Mbar)$.
Consequently, each operator $T_{0j}$ belongs to ${\mathfrak S}_p$, and
so $T_0 = \sum_{j=1}^M T_{0j}$ belongs to ${\mathfrak S}_p$.

In conclusion, $\iota = \sum_{j=0}^N T_j : x^{\gamma}H^s_b(\R^n) \to H^{s'}_b(\R^n)$
belongs to ${\mathfrak S}_p$ (provided that $s > s' + n/p$), and the proof of the
theorem is complete.
\end{proof}

\medskip

For the analysis of boundary value problems we will also need the corresponding
version of Theorem~\ref{EmbeddingbCone} for the appropriate weighted $b$-Sobolev
spaces on certain manifolds with corners.

More precisely, let $\Mbar$ be a compact $n$-manifold with corners of codimension
two (we work with the terminology from \cite{RBM2} here).
Let $\partial\Mbar = \partial_{\reg}\Mbar \cup \partial_{\sing}\Mbar$, where
both $\partial_{\reg}\Mbar$ and $\partial_{\sing}\Mbar$ consist of unions of (different)
boundary hypersurfaces of $\Mbar$. We will refer to those hypersurfaces as regular
or singular, respectively. We require that for any two hypersurfaces $H$
and $H'$ of the boundary with either $H,H' \subset \partial_{\reg}\Mbar$ or
$H,H' \subset \partial_{\sing}\Mbar$ we either have $H\cap H' = \emptyset$ or
$H = H'$. Consequently, the codimension two strata occur as intersections of
regular and singular hypersurfaces only. Both $\partial_{\reg}\Mbar$ and
$\partial_{\sing}\Mbar$ are smooth compact manifolds with
boundary, and we have $\partial\bigl(\partial_{\reg}\Mbar\bigr) =
\partial\bigl(\partial_{\sing}\Mbar\bigr) = \partial_{\reg}\Mbar \cap
\partial_{\sing}\Mbar$.

Let $2\Mbar_{\reg}$ be the double of $\Mbar$ across the regular boundary
hypersurfaces. $2\Mbar_{\reg}$ is a compact smooth manifold with boundary, and
we have $\Mbar \subset 2\Mbar_{\reg}$. Let $r^+$ be the restriction operator
for distributions on the interior of $2\Mbar_{\reg}$ to the interior of $\Mbar$,
and define as usual $H^s_b(\Mbar):= r^+H^s_b(2\Mbar_{\reg})$ equipped with the quotient
topology.
More generally, if $E$ is a (Hermitian) vector bundle on $\Mbar$ and $x$ is
a defining function for $\partial_{\sing}\Mbar$, we obtain the weighted space
$x^{\gamma}H^s_b(\Mbar;E)$ in the way just described.

\begin{corollary}\label{EmbeddingbConeBdry}
The statement of Theorem~\ref{EmbeddingbCone} is valid for the weighted
$H^s_b$-spaces on compact manifolds with corners of codimension two.
\end{corollary}
\begin{proof}
We just need to note that the embedding $x^{\gamma}H^s_b(\Mbar) \hookrightarrow 
x^{\gamma'}H^{s'}_b(\Mbar)$ can be written as the composition
of the maps $r^+\circ\iota\circ e_{s,\gamma}$, where $e_{s,\gamma} :
x^{\gamma}H^s_b(\Mbar) \to x^{\gamma}H^s_b(2\Mbar_{\reg})$ is an
extension operator, $\iota : x^{\gamma}H^s_b(2\Mbar_{\reg}) \to
x^{\gamma'}H^{s'}_b(2\Mbar_{\reg})$ is the embedding that belongs to
${\mathfrak S}_p$ according to Theorem~\ref{EmbeddingbCone} (provided that
$\gamma > \gamma'$ and $s > s' + n/p$ as is assumed here), and
$r^+ : x^{\gamma'}H^{s'}_b(2\Mbar_{\reg}) \to x^{\gamma'}H^{s'}_b(\Mbar)$ is the
restriction operator.
\end{proof}

%%%%%%%%%%%%%%%%%%%%%%%%%%%%%%%%%%%%%%%%%%%%%%%%%%%%%%%%%%%%%%%%%%%%%
\section{Cone operators and rays of minimal growth}\label{ConeOperators}
%%%%%%%%%%%%%%%%%%%%%%%%%%%%%%%%%%%%%%%%%%%%%%%%%%%%%%%%%%%%%%%%%%%%%

\noindent
In this section we compile the definitions and some of the basic results about
cone operators. For detailed accounts we refer to the monograph \cite{Le97} and
the papers \cite{GKM1,GKM2,GiMe01}. Boundary value problems for cone operators
are discussed in \cite{CorSchroSei,KrainBVPCone}. There are many more references
that could be mentioned, but those are the ones that are closest to our present
scope since they emphasize the unbounded operator aspect and discuss operators
of general form.

Since rays of minimal growth are essential in the context of the present paper, we will
proceed to review the results from \cite{GKM1,GKM2,GKM3,KrainBVPCone} about when
a ray $\Gamma \subset \C$ is a ray of minimal growth for a closed extension of an
elliptic cone operator.

\medskip

Let $\Mbar$ be a smooth, compact $n$-manifold with boundary $Y$.
The natural framework for cone geometry is the $c$-cotangent bundle
\begin{equation}\label{cCotangent}
\cpi:\cT^*\Mbar \to \Mbar,
\end{equation}
see~\cite{GKM1}, a vector bundle whose space of smooth sections is in one-to-one 
correspondence with the space of all smooth $1$-forms on $\Mbar$ that are conormal 
to $Y$, i.e., all $\omega \in C^{\infty}(\Mbar,T^*\Mbar)$ whose pullback to $Y$
vanishes. The isomorphism is given by a bundle homomorphism
\begin{equation}\label{TheHomomorphism}
\cev: \cT^*\Mbar \to T^*\Mbar
\end{equation}
which is an isomorphism over $\open \Mbar$.
In coordinates $(x,y_1,\dotsc,y_{n-1})$ near the boundary, where $x$ is a defining 
function for $Y$, a local frame for $\cT^*\Mbar$ is given by the sections mapped by 
$\cev$ to the forms $dx$, $xdy_1, \dotsc, xdy_{n-1}$.

By a $c$-metric we mean any metric on the dual of $\cT^*\Mbar$. Such a metric induces 
(via the homomorphism \eqref{TheHomomorphism}) a Riemannian metric $\cg$ on
$\open\Mbar$. In coordinates near the boundary as in the previous paragraph, $\cg$ is
represented as a smooth symmetric $2$-cotensor
\begin{equation*}
\cg = g_{00}\,dx\otimes dx + \sum\limits_{j=1}^{n-1}g_{0j}\,dx\otimes xdy_j +
\sum\limits_{i=1}^{n-1}g_{i0}\,xdy_i\otimes dx + \sum\limits_{i,j=1}^{n-1}g_{ij}\,xdy_i\otimes xdy_j;
\end{equation*}
the matrix $(g_{ij})$ depends smoothly on $(x,y)$ and is positive definite up to $x=0$.

Special cases of $c$-metrics are warped and straight cone metrics. A warped cone metric 
is a Riemannian metric on $\Mbar$ such that there is a diffeomorphism of a neighborhood 
$U$ of $Y$ in $M$ to $[0,\eps) \times Y$ under which the metric takes on the form
$dx^2 + x^2g_Y(x)$ for a family of metrics $g_Y(x)$ on $Y$ which is smooth up to $x=0$; 
here $x$ is of course the variable in $[0,\eps)$. If the diffeomorphism is such that 
$g_Y(x)$ is in fact independent of $x$ for small $\eps$, then $\cg$ is a straight cone 
metric.

\medskip

Let $E,F \to \Mbar$ be (Hermitian) vector bundles. A cone differential
operator of order $m$ acting from sections of $E$ to sections of $F$
is an element $A$ of $x^{-m}\Diff_b^m(\Mbar;E,F)$, where $\Diff_b^m(\Mbar;E,F)$ is
the space of  totally characteristic differential operators of order $m$, see
Section~\ref{ConeSobolevSpacesEmbeddings}.
Thus $A$ is a linear differential operator $C^\infty(\open \Mbar;E) \to
C^{\infty}(\open \Mbar;F)$, of order $m$, which near any point in $Y$,
in coordinates $(x,y_1,\ldots,y_{n-1})$ as above, is of the form 
\begin{equation}\label{cone-operator}
  A=x^{-m}\sum_{k+|\alpha|\le m} a_{k\alpha}(x,y)(xD_x)^k D_y^\alpha 
\end{equation} 
with coefficients $a_{k\alpha}$ smooth up to $x=0$. For example, the Laplacian
with respect to any $c$-metric is a cone differential operator of order $2$.

The standard principal symbol of a cone operator $A$ over the interior determines, with 
the aid of the map $\cev$ in \eqref{TheHomomorphism}, a smooth homomorphism 
$\cpi^*E\to\cpi^*F$. This is the $c$-principal symbol $\csym(A)$ of $A$. In local 
coordinates near $Y$, 
\begin{equation*}
  \csym(A)=\sum_{k+|\alpha|= m} a_{k\alpha}(x,y)\xi^k \eta^\alpha. 
\end{equation*} 
The operator $A$ is said to be $c$-elliptic if $\csym(A)$ is invertible on
$\cT^*M\minus 0$.

In the sequel we fix an operator $A \in x^{-m}\Diff_b^m(\Mbar;E)$, $m > 0$, and
assume that it is $c$-elliptic.
For every weight $\gamma \in \R$ the operator $A$ is a densely defined
unbounded operator
\begin{equation}\label{StartDomain}
A:C_c^\infty(\open \Mbar;E)\subset x^{\gamma}L_b^2(\Mbar;E)\to x^{\gamma} L_b^2(\Mbar;E).
\end{equation}
Observe that the geometric $L^2$-space with respect to any $c$-metric on $\Mbar$
and Hermitian metric on $E$ is the space $x^{-n/2}L_b^2(\Mbar;E)$, where $n = \dim\Mbar$.

For any choice of $\gamma \in \R$ there are two canonical closed extensions of $A$: 
\begin{align*}
\Dom_{\min} &= \text{domain of the closure of \eqref{StartDomain}},\\
\Dom_{\max} &=\{u\in x^{\gamma}L_b^2(\Mbar;E);\; Au\in x^{\gamma}L_b^2(\Mbar;E)\}.
\end{align*}
These are complete with respect to the graph norm, $\|u\|_A=\|u\|+\|Au\|$, and the 
former is a subspace of the latter.
The following theorem lists basic results proved in \cite{GiMe01,Le97}.

\begin{theorem}\label{BasicTheorem}
Let $A\in x^{-m}\Diff_b^m(\Mbar;E)$, $m > 0$, be $c$-elliptic, and consider $A$ an
unbounded operator in $x^{\gamma}L^2_b(\Mbar;E)$ as described above.
\begin{enumerate}[(a)]
\item $\dim \Dom_{\max}/\Dom_{\min}<\infty$. In particular, every
intermediate space $\Dom_{\min}\subset \Dom\subset \Dom_{\max}$
gives rise to a closed extension 
\begin{equation*}
A_{\Dom}:\Dom\subset x^{\gamma}L^2_b(\Mbar;E) \to x^{\gamma}L^2_b(\Mbar;E).
\end{equation*}
\item All closed extensions $A_{\Dom}$ of $A$ are Fredholm. Moreover, 
\begin{equation}\label{RelIndexA}
\Ind A_{\Dom}= \Ind A_{\Dom_{\min}}+ \dim\Dom/\Dom_{\min}. 
\end{equation}
\item $\Dom_{\min} = \bigcap_{\eps > 0}x^{\gamma+m-\eps}H^m_b(\Mbar;E) \cap
\Dom_{\max}$.

Moreover, $x^{\gamma+m}H^m_b(\Mbar;E) \subset \Dom_{\min}$, and there is equality
$x^{\gamma+m}H^m_b(\Mbar;E) = \Dom_{\min}$ if and only if
$\spec_b(A)\cap\{\sigma \in \C;\; \Im(\sigma) = -\gamma-m\} = \emptyset$.
The set $\spec_b(A) \subset \C$ is the boundary spectrum of $A$, see \cite{RBM2},
a discrete set that contains at most finitely many points in each horizontal strip
of finite width.
\item There exists $\eps > 0$ such that $\Dom_{\max} \hookrightarrow
x^{\gamma+\eps}H^m_b(\Mbar;E)$.
\end{enumerate}
\end{theorem}

By (d) of Theorem~\ref{BasicTheorem} and Theorem~\ref{EmbeddingbCone} we immediately
obtain the following corollary.

\begin{corollary}\label{DomainsSchattenEmb}
Under the assumptions of Theorem~\ref{BasicTheorem} the embedding
$\Dom_{\max} \hookrightarrow x^{\gamma}L^2_b(\Mbar;E)$ belongs to
${\mathfrak S}^+_{n/m}$, see \eqref{Spplus}.

In particular, the embeddings of the domains $\Dom \hookrightarrow
x^{\gamma}L^2_b(\Mbar;E)$ of all closed extensions $A_{\Dom}$ of $A$ in 
$x^{\gamma}L^2_b(\Mbar;E)$ belong to ${\mathfrak S}^+_{n/m}$.
\end{corollary}

In the study of rays of minimal growth for closed extensions of $A$ the
normal operator $A_{\wedge}$ associated with $A$ plays a significant role.
$A_{\wedge}$ is an operator acting in sections on the inward pointing half
of the normal bundle of $Y$ in $\Mbar$.
More precisely, $A_{\wedge}$ is defined as follows:

We first note that any choice of defining function $x$ for $Y$ trivializes
the normal bundle $NY$ to $Y\times\R$. $x$ induces the map $x_{\wedge} = dx$ on
$NY$, and the trivialization $NY \cong Y\times\R$ then is such that $x_{\wedge}$
corresponds to the projection on the second coordinate on $Y\times\R$.
To simplify notation, we will just write $x$ for $x_{\wedge}$ from now on. Let
$Y^{\wedge} = Y\times\overline{\R}_+$ be the inward pointing half of the normal bundle. 
The bundle $E|_Y$ lifts to $Y^{\wedge}$ and carries a natural Hermitian metric and 
connection induced by the metric and connection given on $E$. As is custom in the
literature on cone operators, this bundle on $Y^{\wedge}$ is for sake of simplicity
also denoted by $E$.
On $Y^{\wedge}$ we consider the $b$-density $\frac{dx}{x}\otimes\m_Y$ with a fixed
density $\m_Y$ on $Y$ (lifted to $Y^{\wedge}$).

Choose a collar neighborhood $U$ of $Y$ in $\Mbar$. Pull-back and parallel transport
induce an isomorphism $x^{\gamma}L^2_b(U;E) \cong x^{\gamma}L^2_b(Y\times[0,\eps);E)$
for $\eps > 0$ small enough (choosing the defining function $x$ and the collar
neighborhood $U$ properly even produces a unitary map, but this will not be essential
for us here). Hence, locally near $Y$, $L^2$-sections of $E$ on $\Mbar$ can be
identified with $L^2$-sections of $E$ on $Y^{\wedge}$. This identification extends
to distributional sections and restricts to smooth sections (with compact support).
So, if $A \in x^{-m}\Diff^m_b(\Mbar;E)$, then near $Y$ we can now write
\begin{equation}\label{Arepresentation}
A = x^{-m}\sum\limits_{k=0}^ma_k(x)(xD_x)^k
\end{equation}
with $a_k \in C^{\infty}([0,\eps),\Diff^{m-k}(Y;E|_Y))$.
The normal operator associated with $A$ is the operator
\begin{equation}\label{normalop}
A_{\wedge} = x^{-m}\sum\limits_{k=0}^ma_k(0)(xD_x)^k : C_c^{\infty}(\open Y^{\wedge};E)
\to C^{\infty}(\open Y^{\wedge};E).
\end{equation}
For every $\gamma \in \R$, $A_{\wedge}$ is an unbounded operator
$$
A_{\wedge} : C_c^{\infty}(\open Y^{\wedge};E) \subset x^{\gamma}L^2_b(Y^{\wedge};E) \to
x^{\gamma}L^2_b(Y^{\wedge};E).
$$
Like $A$, $A_{\wedge}$ has the canonical closed minimal and maximal extensions
$\Dom_{\wedge,\min}$ and $\Dom_{\wedge,\max}$.
There is a natural isomorphism
\begin{equation}\label{thetaiso}
\theta : \Dom_{\max}/\Dom_{\min} \to \Dom_{\wedge,\max}/\Dom_{\wedge,\min}
\end{equation}
constructed in \cite{GKM1,GKM2} (and subsequently reviewed in \cite{GKM3,GKM4}).
Without going into further technical details, we just note that the construction of
$\theta$ follows a simple algorithm of $m$ steps, where $m$ is the order of $A$.
It involves the first $m$ Taylor coefficients of the expansions of the $a_k(x)$ in
\eqref{Arepresentation} (that is to say the conormal symbols of the operator $A$
up to order $m$). In the special case that $A$ has constant coefficients, i.e.
the $a_k(x)$ are independent of $x$ for small $x$, we simply have
$\theta\bigl(u + \Dom_{\min}) = \omega u + \Dom_{\wedge,\min}$, where
$\omega \in C_c^{\infty}([0,\eps))$ is a cut-off function near $x = 0$ that
we consider a function on $\Mbar$ supported near $Y$ (this representation of $\theta$
involves passage for functions on $\Mbar$ supported near $Y$ to functions on
$Y^{\wedge}$ as was discussed earlier).

Using \eqref{thetaiso} we can associate with any domain
$\Dom_{\min} \subset \Dom \subset \Dom_{\max}$ for $A$ a domain
$\Dom_{\wedge,\min} \subset \Dom_{\wedge} \subset \Dom_{\wedge,\max}$ for
$A_{\wedge}$ via
\begin{equation}\label{assocdomain}
\Dom_{\wedge}/\Dom_{\wedge,\min} = \theta\bigl(\Dom/\Dom_{\min}\bigr).
\end{equation}
Now let $\Gamma = \{re^{i\theta};\; r \geq 0\} \subset \C$ be a ray. The following
theorem, proved in \cite{GKM2}, gives verifiable criteria for $\Gamma$ to be
a ray of minimal growth for the closed extension $A_{\Dom}$ of an elliptic cone
operator.

\begin{theorem}\label{Rayofminimalgrowth}
Let $A \in x^{-m}\Diff_b^m(\Mbar;E)$, $m > 0$, be $c$-elliptic, and let
$A_{\Dom}$ be any closed extension of $A$ in $x^{\gamma}L^2_b(\Mbar;E)$.
We assume that
\begin{itemize}
\item $A$ is $c$-elliptic with parameter in $\Gamma$, i.e. the $c$-principal
symbol $\csym(A)$ does not have spectrum in $\Gamma$;
\item $\Gamma$ is a ray of minimal growth for the closed extension
$A_{\wedge,\Dom_{\wedge}}$ of the normal operator $A_{\wedge}$ in
$x^{\gamma}L^2_b(Y^{\wedge};E)$, where $\Dom_{\wedge}$ is the associated
domain to $\Dom$ according to \eqref{assocdomain}.
\end{itemize}
Then $\Gamma$ is a ray of minimal growth for $A_{\Dom}$.
\end{theorem}
The second assumption on the normal operator can be phrased conveniently
in geometric terms that involve the action
\begin{equation}\label{kappaaction}
\kappa_{\varrho}u(x,y) = u(\varrho x,y), \varrho > 0,
\end{equation}
that is defined for functions on $Y^{\wedge}$ (for sections of bundles
the definition of this action involves in addition parallel transport in
the fibres). Since both $\Dom_{\wedge,\max}$ and $\Dom_{\wedge,\min}$ are
invariant with respect to this action, it descends to an action on the quotient
$\Dom_{\wedge,\max}/\Dom_{\wedge,\min}$ and therefore induces flows on the
various Grassmannians of its subspaces.

For any domain $\Dom_{\wedge}$ let $\Omega^{-}(\Dom_{\wedge})$ consist of
all domains $\tilde{\Dom}_{\wedge}$ of closed extensions of $A_{\wedge}$ such that
$\dim\tilde{\Dom}_{\wedge}/\Dom_{\wedge,\min} = \dim\Dom_{\wedge}/\Dom_{\wedge,\min}$,
so these quotient spaces belong to the same Grassmannian, and such that there exists a
sequence $\varrho_k \to 0$ such that
$$
\kappa_{\varrho_k}\bigl(\Dom_{\wedge}/\Dom_{\wedge,\min}\bigr) \to
\tilde{\Dom}_{\wedge}/\Dom_{\wedge,\min} \textup{ as } k \to \infty
$$
in that Grassmannian. It was shown in \cite{GKM5} that $\Omega^{-}(\Dom_{\wedge})$
has topologically the structure of an embedded torus.

Now, provided that $A$ is $c$-elliptic with parameter in $\Gamma$, it was proved
in \cite{GKM1,GKM3} that $\Gamma$ is a ray of minimal growth for $A_{\wedge}$
with domain $\Dom_{\wedge}$ if and only if, for some $\lambda_0 \in \Gamma$,
$A_{\wedge} - \lambda_0 : \tilde{\Dom}_{\wedge} \to x^{\gamma}L^2_b(Y^{\wedge};E)$
is invertible for all $\tilde{\Dom}_{\wedge} \in \Omega^{-}(\Dom_{\wedge})$.

\bigskip

\noindent
Let us now proceed with the corresponding discussion for realizations of
elliptic boundary value problems on manifolds with conical singularities.

Let $\Mbar$ be a compact $n$-manifold with corners of codimension two, and
let $\partial\Mbar = \partial_{\reg}\Mbar \cup \partial_{\sing}\Mbar$, where
both $\partial_{\reg}\Mbar$ and $\partial_{\sing}\Mbar$ are smooth manifolds
with boundary as described in Section~\ref{ConeSobolevSpacesEmbeddings} before
Corollary~\ref{EmbeddingbConeBdry}.
Let $2\Mbar_{\reg}$ be the double of $\Mbar$ across $\partial_{\reg}\Mbar$.
$2\Mbar_{\reg}$ is a smooth compact manifold with boundary $2\partial_{\sing}\Mbar$,
the double of $\partial_{\sing}\Mbar$ across its boundary.

We obtain the relevant objects on $\Mbar$ by restriction of the corresponding
objects from $2\Mbar_{\reg}$. For example, the $c$-cotangent bundle
$\cT^*\Mbar$ is by definition $\cT^*2\Mbar_{\reg}|_{\Mbar}$.
Similarly, we consider (Hermitian) vector bundles $E \to \Mbar$ that are restrictions
of (Hermitian) vector bundles from $2\Mbar_{\reg}$.
Let $x$ be a defining function for $2\partial_{\sing}\Mbar$ in $2\Mbar_{\reg}$.
For any $m \in \N_0$ and vector bundles $E$ and $F$, let
$x^{-m}\Diff_b^m(\Mbar;E,F)$ be the space of cone differential operators of order $m$
on $\Mbar$ acting from sections of the bundle $E$ to sections of $F$.
Every operator in this space is obtained by restricting a corresponding cone
differential operator from the double $2\Mbar_{\reg}$ to $\Mbar$.

In the sequel, we fix an operator $A \in x^{-m}\Diff_b^m(\Mbar;E)$, $m > 0$, and
a collection of operators $B_j \in x^{-m_j}\Diff_b^{m_j}(\Mbar;E,F_j)$,
$m_j < m$, $j = 1,\ldots,N$, and consider the following boundary value problem
with spectral parameter $\lambda \in \Gamma = \{re^{i\theta};\;r \geq 0\} \subset \C$:
\begin{equation}
\left.\begin{gathered}
(A - \lambda)u = f \textup{ in } \open\Mbar, \\
Tu = \begin{pmatrix} {\mathfrak r}_{\partial_{\reg}\Mbar}\circ B_1u \\ \vdots \\
{\mathfrak r}_{\partial_{\reg}\Mbar}\circ B_Nu \end{pmatrix} = 0 \textup{ on }
\partial_{\reg}\Mbar,
\end{gathered}\right\}
\end{equation}
where ${\mathfrak r}_{\partial_{\reg}\Mbar} : v \mapsto v|_{\partial_{\reg}\Mbar}$ is
the trace operator. More precisely, for any given weight $\gamma \in \R$,
we consider the spectral problem for the operator $A_T$ in
$x^{\gamma}L^2_b(\Mbar;E)$ that acts like $A$ with domain $\Dom(A_T)$,
where $\Dom(A_T)$ is any intermediate space
$\Dom_{\min}(A_T) \subset \Dom(A_T) \subset \Dom_{\max}(A_T)$, and
\begin{align*}
\Dom_{\max}(A_T) &= \{u \in x^{\gamma}H^m_b(\Mbar;E);\; Au \in x^{\gamma}L^2_b(\Mbar;E)
\textup{ and } Tu = 0 \textup{ on } \partial_{\reg}\Mbar\}, \\
\Dom_{\min}(A_T) &= \Dom_{\max}(A_T) \cap
\bigcap_{\varepsilon > 0}x^{\gamma+m-\varepsilon}H^m_b(\Mbar;E).
\end{align*}
We henceforth assume that
\begin{equation}\label{AssumptionsBVPs}
\left.
\mbox{
\parbox{11cm}{
\begin{itemize}
\item $A$ is $c$-elliptic with parameter in $\Gamma$, i.e., $\csym(A) - \lambda$ is
invertible everywhere on $\bigl(\cT^*\Mbar\times\Gamma\bigr)\setminus 0$;
\item the $c$-principal boundary symbol with parameter
$$
\begin{pmatrix} \csym_{\partial}(A) - \lambda \\ \csym_{\partial}(T) \end{pmatrix} :
{}^c{\mathscr S}_+ \otimes \cpi^*E|_{\partial_{\reg}\Mbar} \to
\begin{array}{c}
{}^c{\mathscr S}_+ \otimes \cpi^*E|_{\partial_{\reg}\Mbar} \\
\oplus \\
\bigoplus_{j=1}^N \cpi^*F_j|_{\partial_{\reg}\Mbar}
\end{array}
$$
is invertible on $\bigl(\cT^*\partial_{\reg}\Mbar \times \Gamma\bigr) \setminus 0$,
where $\cpi : \cT^*\partial_{\reg}\Mbar \to \partial_{\reg}\Mbar$ is the canonical
projection.
\end{itemize}
}
}
\right\}
\end{equation}
We proceed to explain the notion of $c$-principal boundary symbol from
\eqref{AssumptionsBVPs}.
Let $y_1$ be a defining function for $\partial_{\reg}\Mbar$ such that
$dx\wedge dy_1 \neq 0$ on $\partial_{\reg}\Mbar \cap \partial_{\sing}\Mbar$.
In local coordinates near $\partial_{\reg}\Mbar$ write
$$
\csym(A) = \sum_{j+|\alpha|=m}a_{j,\alpha}(y',y_1)\eta'^{\alpha}\eta_1^j.
$$
Then the $c$-principal boundary symbol of $A$ is
$$
\csym_{\partial}(A) =
\sum_{j+|\alpha|=m}a_{j,\alpha}(y',0)\eta'^{\alpha}D_{y_1}^j :
{\mathscr S}(\overline{\mathbb R}_+)\otimes\C^K \to
{\mathscr S}(\overline{\mathbb R}_+)\otimes\C^K,
$$
where $K = \dim E$. Globally this leads to
$$
\csym_{\partial}(A) : {}^c{\mathscr S}_+ \otimes \cpi^*E|_{\partial_{\reg}\Mbar} \to
{}^c{\mathscr S}_+ \otimes \cpi^*E|_{\partial_{\reg}\Mbar},
$$
where ${}^c{\mathscr S}_+ \to \cT^*\partial_{\reg}\Mbar$ is a vector bundle
with fiber ${\mathscr S}(\overline{\mathbb R}_+)$.

Analogously, we have
$$
\csym_{\partial}(B_j) : {}^c{\mathscr S}_+ \otimes \cpi^*E|_{\partial_{\reg}\Mbar} \to
{}^c{\mathscr S}_+ \otimes \cpi^*F_j|_{\partial_{\reg}\Mbar}, \quad
j = 1,\ldots,N.
$$
Let $\csym_{\partial}({\mathfrak r}_{\partial_{\reg}\Mbar}) : u \to u(0)$ fiberwise
in ${}^{c}{\mathscr S}_+$. Combined this gives
$$
\csym_{\partial}(T) = \begin{pmatrix}
\csym_{\partial}({\mathfrak r}_{\partial_{\reg}\Mbar})\circ
\csym_{\partial}(B_1) \\ \vdots \\
\csym_{\partial}({\mathfrak r}_{\partial_{\reg}\Mbar})\circ
\csym_{\partial}(B_N) \end{pmatrix} :
{}^c{\mathscr S}_+ \otimes \cpi^*E|_{\partial_{\reg}\Mbar} \to
\bigoplus_{j=1}^N \cpi^*F_j|_{\partial_{\reg}\Mbar},
$$
the $c$-principal boundary symbol of the boundary condition $T$.

\medskip

For the discussion of rays of minimal growth, we also need to impose
a parameter-dependent ellipticity condition that is associated with
$\overline{Y} = \partial_{\sing}\Mbar$.
This condition involves the normal operator $A_{\wedge}$ of $A$ and a corresponding
normal boundary value problem $T_{\wedge}$ for $A_{\wedge}$ on $\overline{Y}^{\wedge}$.
The definition of $A_{\wedge}$ is exactly like in \eqref{normalop}. Likewise, there
are normal operators $B_{j,\wedge}$ associated with the operators $B_j$, $j=1,\ldots,N$.
The normal operator associated with $T$ is then
$$
T_{\wedge} =
\begin{pmatrix}
{\mathfrak r_{(\partial\overline{Y})^{\wedge}}} \circ B_{1,\wedge}
\\
\vdots
\\
{\mathfrak r_{(\partial\overline{Y})^{\wedge}}} \circ B_{N,\wedge}
\end{pmatrix} :
C_c^{\infty}(\overline{Y}^{\wedge};E) \to
C_c^{\infty}((\partial\overline{Y})^{\wedge};\bigoplus_{j=1}^NF_j).
$$
For the previously fixed weight $\gamma \in \R$ we consider the spectral problem for
the realizations of $A_{\wedge}$ subject to $T_{\wedge}u = 0$ in
$x^{\gamma}L^2_b(\overline{Y}^{\wedge};E)$. More precisely, we consider the operator
$A_{\wedge,T_{\wedge}}$ that acts like $A_{\wedge}$ with domain
$\Dom_{\wedge}(A_{\wedge,T_{\wedge}})$, where
$\Dom_{\wedge}(A_{\wedge,T_{\wedge}})$ is any intermediate space
$\Dom_{\wedge,\min}(A_{\wedge,T_{\wedge}}) \subset \Dom_{\wedge}(A_{\wedge,T_{\wedge}})
\subset \Dom_{\wedge,\max}(A_{\wedge,T_{\wedge}})$. Here
\begin{align*}
\Dom_{\wedge,\max}(A_{\wedge,T_{\wedge}}) &= \{u \in
\K^{m,\gamma}(\overline{Y}^{\wedge};E)_{\gamma};\; A_{\wedge}u \in x^{\gamma}L^2_b(\overline{Y}^{\wedge};E) \textup{ and } T_{\wedge}u = 0\}, \\
\Dom_{\wedge,\min}(A_{\wedge,T_{\wedge}}) &= \Dom_{\wedge,\max}(A_{\wedge,T_{\wedge}})
\cap \bigcap_{\varepsilon > 0}
\K^{m,\gamma+m-\varepsilon}(\overline{Y}^{\wedge};E)_{\gamma},
\end{align*}
and for $s,\delta,\delta' \in \R$,
$$
\K^{s,\delta}(\overline{Y}^{\wedge})_{\delta'} =
\omega x^{\delta}H^s_b(\overline{Y}^{\wedge}) +
(1-\omega) x^{\delta'+n/2}H^s_{\textup{cone}}(\overline{Y}^{\wedge})
$$
is a weighted cone Sobolev space on $\overline{Y}^{\wedge}$, see
\cite{SchulzeNH,EgorovKondratievSchulze2,KrainBVPCone}. Here
$\omega \in C_c^{\infty}(\overline{\R}_+)$ is a cut-off function near zero.

It was shown in \cite{KrainBVPCone} that under our present assumptions
\eqref{AssumptionsBVPs} there exists a natural isomorphism
$$
\theta : \Dom_{\max}(A_T)/\Dom_{\min}(A_T) \to
\Dom_{\wedge,\max}(A_{\wedge,T_{\wedge}})/\Dom_{\wedge,\min}(A_{\wedge,T_{\wedge}})
$$
similar to \eqref{thetaiso} that allows passage from domains $\Dom(A_T)$
of realizations of $A$ subject to $Tu = 0$ to associated domains
$\Dom_{\wedge}(A_{\wedge,T_{\wedge}})$ of realizations of $A_{\wedge}$ subject
to $T_{\wedge}u = 0$ via
\begin{equation}\label{assocdomain1}
\Dom_{\wedge}(A_{\wedge,T_{\wedge}})/\Dom_{\wedge,\min}(A_{\wedge,T_{\wedge}}) =
\theta\bigl(\Dom(A_T)/\Dom_{\min}(A_T)\bigr),
\end{equation}
see also \eqref{assocdomain}. Moreover, the quotient spaces
$\Dom_{\max}(A_T)/\Dom_{\min}(A_T)$ and correspondingly
$\Dom_{\wedge,\max}(A_{\wedge,T_{\wedge}})/\Dom_{\wedge,\min}(A_{\wedge,T_{\wedge}})$
are finite dimensional.

\medskip

In addition to \eqref{AssumptionsBVPs} we will require the following
parameter-dependent ellipticity condition associated with $\partial_{\sing}\Mbar$:
\begin{equation}\label{ParamEllSingBdry}
\mbox{
\parbox{9cm}{
The ray $\Gamma$ is a ray of minimal growth for $A_{\wedge,T_{\wedge}}$ with
the associated domain $\Dom_{\wedge}(A_{\wedge,T_{\wedge}})$ to $\Dom(A_T)$
according to \eqref{assocdomain1}.}
}
\end{equation}
The following theorem is the main result of \cite{KrainBVPCone}.

\begin{theorem}\label{RayofminimalgrowthBVP}
Consider the realization $A_T$ of $A$ subject to the boundary condition $Tu = 0$
on $\Mbar$ in $x^{\gamma}L^2_b(\Mbar;E)$ with domain $\Dom(A_T)$, where
$\Dom_{\min}(A_T) \subset \Dom(A_T) \subset \Dom_{\max}(A_T)$, and let
$\Gamma = \{re^{i\theta};\; r \geq 0\} \subset \C$ be a ray.
Assume that the parameter-dependent ellipticity conditions \eqref{AssumptionsBVPs}
and \eqref{ParamEllSingBdry} are fulfilled.

Then $\Gamma$ is a ray of minimal growth for the operator $A_T : \Dom(A_T)
\to x^{\gamma}L^2_b(\Mbar;E)$.
\end{theorem}

Under the assumptions of Theorem~\ref{RayofminimalgrowthBVP} it was shown
in \cite{KrainBVPCone} that \emph{all} realizations of $A_T$ with domains
between $\Dom_{\min}(A_T)$ and $\Dom_{\max}(A_T)$ are closed operators in
the functional analytic sense, that they are all Fredholm, and, moreover,
that $\Dom_{\max}(A_T) \hookrightarrow x^{\gamma+\varepsilon}H^m_b(\Mbar;E)$ for
a sufficiently small $\varepsilon > 0$.
In view of Corollary~\ref{EmbeddingbConeBdry}, the latter implies the following.

\begin{corollary}\label{EmbeddingDomainsBVP}
Under the assumptions of Theorem~\ref{RayofminimalgrowthBVP}, the embedding
of the domain $\Dom(A_T) \hookrightarrow x^{\gamma}L^2_b(\Mbar;E)$ belongs
to ${\mathfrak S}^+_{n/m}$, see \eqref{Spplus}.
\end{corollary}

Finally, we note that the assumption \eqref{ParamEllSingBdry} can be checked
effectively using the dilation group $\kappa_{\varrho}$ from \eqref{kappaaction}
and the induced flow on the Grassmannians of subspaces of the quotient
$\Dom_{\wedge,\max}(A_{\wedge,T_{\wedge}})/\Dom_{\wedge,\min}(A_{\wedge,T_{\wedge}})$
analogously to the case of closed extensions of cone operators without
boundary conditions, see the explanation after Theorem~\ref{Rayofminimalgrowth}.

%%%%%%%%%%%%%%%%%%%%%%%%%%%%%%%%%%%%%%%%%%%%%%%%%%%%%%%%%%%%%%%%%%%%%
\section{Main theorems and examples}\label{MainTheoremExamples}
%%%%%%%%%%%%%%%%%%%%%%%%%%%%%%%%%%%%%%%%%%%%%%%%%%%%%%%%%%%%%%%%%%%%%

\noindent
What remains to be done is to combine the results from the previous sections
to obtain our main Theorems~\ref{MainTheorem} and \ref{MainTheoremBVP} about
the completeness of generalized eigenfunctions for elliptic cone operators.

\begin{theorem}\label{MainTheorem}
Let $\Mbar$ be a smooth compact $n$-manifold with boundary $Y$, and let
$A \in x^{-m}\Diff^m_b(\Mbar;E)$, $m > 0$, be $c$-elliptic.
Fix a weight $\gamma \in \R$, and consider the closed extension
$$
A_{\Dom} : \Dom \subset x^{\gamma}L^2_b(\Mbar;E) \to x^{\gamma}L^2_b(\Mbar;E)
$$
of $A$.
We assume that there are rays
$$
\Gamma_j = \{re^{i\theta_j};\; r \geq 0\}, \quad j = 1,\ldots,J,
$$
in the complex plane such that all angles enclosed by any two adjacent rays are
$\leq \frac{\pi m}{n}$, and such that for any such ray $\Gamma$,
\begin{itemize}
\item $\csym(A) - \lambda$ is invertible on
$\bigl(\cT^*\Mbar\times\Gamma\bigr)\setminus 0$;
\item $\Gamma$ is a ray of minimal growth for
$$
A_{\wedge} : \Dom_{\wedge} \subset x^{\gamma}L^2_b(Y^{\wedge};E) \to
x^{\gamma}L^2_b(Y^{\wedge};E)
$$
for the associated domain $\Dom_{\wedge}$ to $\Dom$ according to \eqref{assocdomain}.
\end{itemize}
Then the system of generalized eigenfunctions of $A_{\Dom}$ is complete in
$x^{\gamma}L^2_b(\Mbar;E)$.
\end{theorem}

As was pointed out after Theorem~\ref{Rayofminimalgrowth}, we note that the assumption
that $\Gamma$ be a ray of minimal growth for $A_{\wedge}$ with domain $\Dom_{\wedge}$
can be checked effectively using the dilation group $\kappa_{\varrho}$ from
\eqref{kappaaction} and the induced flow on the Grassmannian of subspaces of
the quotient $\Dom_{\wedge,\max}/\Dom_{\wedge,\min}$ that contains
the subspace $\Dom_{\wedge}/\Dom_{\wedge,\min}$. We will illustrate this
in Example~\ref{Example} below.

\medskip

\noindent
In the special case where $\Dom = \Dom_{\min} = x^{\gamma+m}H^m_b(\Mbar;E)$,
Theorem~\ref{MainTheorem} was proved by Egorov, Kondratiev, and Schulze in
\cite{EgorovKondratievSchulze1}. The following example illustrates why
the general result is relevant.
Further information pertaining to this example can be found in
\cite{GiMe01} (as far as the Friedrichs domain is concerned), and, in particular,
in \cite{GKM3}.

\begin{example}\label{Example}
Let $\Mbar$ be a smooth compact $2$-manifold with boundary $Y = {\mathbb S}^1$.
Fix a collar neighborhood map $U \cong Y \times [0,\varepsilon)$ of the boundary,
and a defining function $x$ for $Y$ that coincides in $U$ with the projection to
the coordinate in $[0,\varepsilon)$.
Let ${}^cg$ be a Riemannian metric on $\open \Mbar$ that in the splitting of variables
$(y,x) \in Y\times[0,\varepsilon)$ near the boundary takes the form
${}^cg = dx^2 + x^2g_Y(x)$ for a smooth family of metrics $g_Y(x)$ on $Y$ up to $x = 0$,
and assume that $g_Y(0)$ is the standard round metric on ${\mathbb S}^1$.

${}^cg$ is a special $c$-metric as was discussed at the beginning of
Section~\ref{ConeOperators}, and the positive Laplacian $\Delta = \Delta_{{}^cg} \in
x^{-2}\Diff^2_b(\Mbar)$ is a cone differential operator. Its $c$-principal symbol
$\csym(\Delta)$ is the metric induced by ${}^cg$ on $\cT^*\Mbar$. Consequently,
$\csym(\Delta)-\lambda$ is invertible for all $\lambda \notin \overline{\R}_+$,
i.e., $\Delta$ is $c$-elliptic with parameter $\lambda \in \Gamma$ for \emph{all}
rays $\Gamma \neq \overline{\R}_+$.

The geometric $L^2$-space with respect to the metric ${}^cg$ is the space
$x^{-1}L^2_b(\Mbar)$, and we consider $\Delta$ an unbounded operator
$$
\Delta : C_c^{\infty}(\open\Mbar) \subset x^{-1}L^2_b(\Mbar) \to x^{-1}L^2_b(\Mbar).
$$
$\Delta$ has infinitely many selfadjoint and infinitely many nonselfadjoint
closed extensions.
In fact, $\dim\Dom_{\max}/\Dom_{\min} = 2$, $\Ind\Delta_{\min} = -1$, and
$\Ind\Delta_{\max} = 1$. The domains $\Dom$ of closed extensions of $\Delta$
such that $\Ind\Delta_{\Dom} = 0$ are the ones with $\dim\Dom/\Dom_{\min} = 1$.
Using Theorem~\ref{MainTheorem}, we will proceed to argue that the system of
generalized eigenfunctions of $\Delta_{\Dom}$ is complete in $x^{-1}L^2_b(\Mbar)$
for all domains $\Dom$ with $\dim\Dom/\Dom_{\min} = 1$. In particular,
this includes all selfadjoint extensions (where the statement is trivial in view
of the spectral theorem), but also infinitely many more nonselfadjoint extensions
of $\Delta$.

The normal operator $\Delta_{\wedge}$ associated to $\Delta_{{}^cg}$ on
$Y^{\wedge} = {\mathbb S}^1 \times \overline{\R}_+$ is the positive Laplacian with
respect to the metric $dx^2 + x^2g_Y(0)$. In other words, it is the standard positive
Laplacian in $\R^2\setminus\{0\}$ in polar coordinates. Correspondingly, the space
$x^{-1}L^2_b(Y^{\wedge})$ is just the standard $L^2$-space on
${\mathbb R}^2\setminus\{0\}$ with respect to Lebesgue measure, written in polar
coordinates. We have
$$
\Dom_{\wedge,\max} = \Dom_{\wedge,\min} \oplus \Span\{\omega,\omega\log x\},
$$
where $\omega \in C_c^{\infty}(\overline{\R}_+)$ is a cut-off function near zero.
This gives an isomorphism
$$
\Dom_{\wedge,\max}/\Dom_{\wedge,\min} \cong \Span\{1,\log x\},
$$
and the action $\kappa_{\varrho}$ from \eqref{kappaaction} that is
induced on $\Dom_{\wedge,\max}/\Dom_{\wedge,\min}$ is given by
$\kappa_{\varrho}1 = 1$ and $\kappa_{\varrho}\log x =
\log(\varrho)\cdot 1 + \log x$ on the basis elements under this isomorphism.

Now let $\Dom_{\wedge}$ be any domain for $\Delta_{\wedge}$ with
$\dim\Dom_{\wedge}/\Dom_{\wedge,\min} = 1$. Then $\Dom_{\wedge}/\Dom_{\wedge,\min}$
corresponds to $\Span\{a\cdot 1 + b\cdot\log x\}$ for some $(a,b) \neq (0,0)$.
$\kappa_{\varrho}$ induces a flow on the Grassmannian of all subspaces
$\tilde{\Dom}_{\wedge}/\Dom_{\wedge,\min}$ of $\Dom_{\wedge,\max}/\Dom_{\wedge,\min}$
with $\dim\tilde{\Dom}_{\wedge}/\Dom_{\wedge,\min} = 1$. In that Grassmannian we have
with the obvious identifications as $\varrho \to 0$
\begin{gather*}
\kappa_{\varrho}\bigl(\Dom_{\wedge}/\Dom_{\wedge,\min}\bigr) =
\Span\{(a+b\log(\varrho))\cdot 1 + b\cdot \log x\}  \\
= \Span\{1 + \frac{b}{a+b\log(\varrho)}\cdot \log x\}
\underset{\varrho \to 0}{\longrightarrow}
\Span\{1\} = \Dom_{\wedge,F}/\Dom_{\wedge,\min},
\end{gather*}
where $\Dom_{\wedge,F}$ is the domain of the Friedrichs extension of $\Delta_{\wedge}$.
This shows that $\Omega^{-}(\Dom_{\wedge}) = \{\Dom_{\wedge,F}\}$ for any domain
$\Dom_{\wedge}$ with $\dim\Dom_{\wedge}/\Dom_{\wedge,\min} = 1$.
Because $\Delta_{\wedge} - \lambda : \Dom_{\wedge,F} \to x^{-1}L^2_b(Y^{\wedge})$
is invertible for all $\lambda \notin \overline{\R}_+$, we conclude that \emph{all}
rays $\Gamma \neq \overline{\R}_+$ are rays of minimal growth for \emph{all} extensions
of $\Delta_{\wedge}$ with domains $\Dom_{\wedge}$ such that
$\dim\Dom_{\wedge}/\Dom_{\wedge,\min} = 1$.

The arguments above now show that Theorem~\ref{MainTheorem} is applicable for all
closed extensions $\Delta_{\Dom}$ in $x^{-1}L^2_b(\Mbar)$ for
\emph{all} domains $\Dom_{\min} \subset \Dom \subset \Dom_{\max}$ with
$\dim\Dom/\Dom_{\min} = 1$. Hence the system of generalized eigenfunctions of
$\Delta_{\Dom}$ is complete in $x^{-1}L^2_b(\Mbar)$ for all these extensions.

This example is clearly not covered by \cite{EgorovKondratievSchulze1}:
Because $\Ind\Delta_{\min} = -1$, the minimal extension of the Laplacian does
not admit any rays of minimal growth. Likewise, $\Delta_{\wedge,\min}$ does
not admit any rays of minimal growth.
Moreover, in this example we also have $x^{1}H^2_b(\Mbar) \subsetneq \Dom_{\min}$
(and the former is of infinite codimension in the latter), which shows that the
scale of weighted $b$-Sobolev spaces that is widely used in the literature on
cone operators cannot be expected to fit into the natural functional analytic
framework of domains of closed extensions.
\end{example}

\begin{theorem}\label{MainTheoremBVP}
Let $\Mbar$ be a compact $n$-manifold with corners of codimension two,
$\partial\Mbar = \partial_{\reg}\Mbar \cup \partial_{\sing}\Mbar$. Let $x$
be a defining function for $\overline{Y} = \partial_{\sing}\Mbar$, and let
$A \in x^{-m}\Diff_b^m(\Mbar;E)$, $m > 0$. Let $T$ be a vector of boundary
conditions for $A$ associated with $\partial_{\reg}\Mbar$.
Fix a weight $\gamma \in \R$, and consider the realization
$$
A_{T,\Dom} : \Dom \subset x^{\gamma}L^2_b(\Mbar;E) \to x^{\gamma}L^2_b(\Mbar;E)
$$
of $A$ subject to $Tu = 0$ on $\partial_{\reg}\Mbar$ with domain
$\Dom_{\min}(A_T) \subset \Dom \subset \Dom_{\max}(A_T)$.
We assume that there are rays
$$
\Gamma_j = \{re^{i\theta_j};\; r \geq 0\}, \quad j = 1,\ldots,J,
$$
in the complex plane such that all angles enclosed by any two adjacent rays are
$\leq \frac{\pi m}{n}$, and such that for any such ray $\Gamma$,
\begin{itemize}
\item $\csym(A) - \lambda$ is invertible on
$\bigl(\cT^*\Mbar\times\Gamma\bigr)\setminus 0$;
\item
$$
\begin{pmatrix} \csym_{\partial}(A) - \lambda \\ \csym_{\partial}(T) \end{pmatrix} :
{}^c{\mathscr S}_+ \otimes \cpi^*E|_{\partial_{\reg}\Mbar} \to
\begin{array}{c}
{}^c{\mathscr S}_+ \otimes \cpi^*E|_{\partial_{\reg}\Mbar} \\
\oplus \\
\bigoplus_{j=1}^N \cpi^*F_j|_{\partial_{\reg}\Mbar}
\end{array}
$$
is invertible on $\bigl(\cT^*\partial_{\reg}\Mbar \times \Gamma\bigr) \setminus 0$,
where $\cpi : \cT^*\partial_{\reg}\Mbar \to \partial_{\reg}\Mbar$ is the canonical
projection;
\item $\Gamma$ is a ray of minimal growth for the realization
$$
A_{\wedge,T_{\wedge}} : \Dom_{\wedge} \subset x^{\gamma}L^2_b(\overline{Y}^{\wedge};E)
\to x^{\gamma}L^2_b(\overline{Y}^{\wedge};E)
$$
of $A_{\wedge}$ subject to $T_{\wedge}u = 0$ with the associated domain
$\Dom_{\wedge}$ to $\Dom$ according to \eqref{assocdomain1}.
\end{itemize}
Then the system of generalized eigenfunctions of $A_{T,\Dom}$ is complete in
$x^{\gamma}L^2_b(\Mbar;E)$.
\end{theorem}

In the special case where
$$
\Dom = \Dom_{\min}(A_T) = \{u \in x^{\gamma+m}H^m_b(\Mbar;E);\; Tu = 0\},
$$
Theorem~\ref{MainTheoremBVP} was obtained by Egorov, Kondratiev, and Schulze in
\cite{EgorovKondratievSchulze2}. The following example illustrates why
the result is relevant in the general case.

\begin{example}
Let $\overline{\Omega} \subset \R^2$ be a bounded domain. We assume that
$\partial\overline{\Omega}\setminus\{0\}$ is $C^{\infty}$, and that the point $0$ is
an angular singularity. More specifically, after rotation, we assume that there is an
angular domain $V = \{z \in \C;\; z = xe^{i\theta},\;x \geq 0,\;
0 \leq \theta \leq \alpha\}$, where $0 < \alpha < 2\pi$, such that
there exists an $\varepsilon > 0$ with
$B_{\varepsilon}(0)\cap\overline{\Omega} = B_{\varepsilon}(0)\cap V$.

In $\overline{\Omega}$ we consider the positive Laplacian
$\Delta = D_{x_1}^2 + D_{x_2}^2$ subject to homogeneous Dirichlet boundary conditions
on $\partial\overline{\Omega}\setminus\{0\}$. We are interested in closed
extensions of this operator in $L^2(\overline{\Omega})$.

By introducing polar coordinates $(x,\theta)$ near $0$, where $x \geq 0$ and
$0 \leq \theta \leq \alpha$, we blow up the origin and obtain a manifold
$\Mbar$ with corners of codimension two. The blow-down map takes
$\Mbar \to \overline{\Omega}$, $\partial_{\sing}\Mbar \to 0$, and
$\partial_{\reg}\Mbar\setminus\partial_{\sing}\Mbar \to
\partial\overline{\Omega}\setminus\{0\}$.

$\Delta$ induces a cone operator on $\Mbar$, and the boundary condition is
the homogeneous Dirichlet boundary condition on $\partial_{\reg}\Mbar$.
The radial variable $x$ gives rise to a defining function for $\partial_{\sing}\Mbar$.
Near $\partial_{\sing}\Mbar$, we have
$\Delta = x^{-2}\bigl((xD_x)^2 + D_{\theta}^2\bigr)$.
We will henceforth write $\Delta_{\textup{Dir}}$ for this operator to emphasize
that it is equipped with Dirichlet boundary conditions.

We consider $\Delta_{\textup{Dir}}$ an unbounded operator in $x^{-1}L^2_b(\Mbar)$.
Observe that the blow-down map takes this space to $L^2(\overline{\Omega})$, the
space we are interested in.

The wealth of extensions of $\Delta_{\textup{Dir}}$ depends strongly on the
angle $\alpha$. More precisely, if $0 < \alpha < \pi$, then
$$
\Dom_{\min}(\Delta_{\textup{Dir}}) = \Dom_{\max}(\Delta_{\textup{Dir}}) =
\{u \in x^1H^2_b(\Mbar);\; u = 0 \textup{ on } \partial_{\reg}\Mbar\}.
$$
If $\alpha = \pi$ (the case when the entire boundary of $\overline{\Omega}$
is smooth), then still
$$
\Dom_{\min}(\Delta_{\textup{Dir}}) = \Dom_{\max}(\Delta_{\textup{Dir}}) \widehat{=}
H^2(\overline{\Omega})\cap H^1_0(\overline{\Omega}),
$$
but this space contains $\{u \in x^1H^2_b(\Mbar);\; u = 0 \textup{ on }
\partial_{\reg}\Mbar\}$ as a proper subspace of infinite codimension. This provides
another simple example that shows that the scale of weighted $b$-Sobolev spaces does not
necessarily fit into the natural functional analytic framework of domains of closed
extensions of an operator.

Consequently, whenever $0 < \alpha \leq \pi$, we have
$\Dom_{\min}(\Delta_{\textup{Dir}}) = \Dom_{\max}(\Delta_{\textup{Dir}})$,
the domain of the Friedrichs extension of the Laplacian. Thus the span of the
eigenfunctions is dense by the spectral theorem.

The situation is more interesting for $\pi < \alpha < 2\pi$. In this case,
$$
\Dom_{\max}(\Delta_{\textup{Dir}}) = \Dom_{\min}(\Delta_{\textup{Dir}}) \oplus
\Span\{\omega(x) \varphi(\theta)x^{\pi/\alpha},\omega(x)\varphi(\theta)x^{-\pi/\alpha}\},
$$
where $\omega \in C_c^{\infty}(\overline{\R}_+)$ is a cut-off function supported near
the origin, and $\varphi(\theta) = \sin\bigl((\pi/\alpha)\theta\bigr)$
is an eigenfunction of $D_{\theta}^2$ to the eigenvalue $(\pi/\alpha)^2$ on the
interval $[0,\alpha]$ subject to Dirichlet boundary conditions (we are using polar
coordinates here as above).
Similarly to Example~\ref{Example}, we will show using Theorem~\ref{MainTheoremBVP}
that the system of generalized eigenfunctions of $\Delta_{\textup{Dir}}$ is complete
in $x^{-1}L^2_b(\Mbar)$ for all domains $\Dom \subset \Dom_{\max}(\Delta_{\textup{Dir}})$
such that $\dim\Dom/\Dom_{\min}(\Delta_{\textup{Dir}}) = 1$.
This includes infinitely many selfadjoint and, most importantly, nonselfadjoint
extensions where the statement is nontrivial.

Clearly, $\Delta_{\textup{Dir}}$ is $c$-elliptic with parameter $\lambda \in \Gamma$
for all rays $\Gamma \neq \overline{\R}_+$, and, likewise, the $c$-principal boundary
symbol with parameter $\lambda \in \Gamma$ is invertible for all these rays
$\Gamma$. In other words, the first two bulleted assumptions of
Theorem~\ref{MainTheoremBVP} are satisfied for $\Delta_{\textup{Dir}}$ for all
rays $\Gamma \neq \overline{\R}_+$. In order to apply Theorem~\ref{MainTheoremBVP},
we need to check the remaining assumptions on the normal operator.
The normal operator is the positive Dirichlet Laplacian $\Delta_{\wedge,\textup{Dir}}$ on
the angular domain $V$, written in polar coordinates. The $L^2$-realizations satisfy
$$
\Dom_{\wedge,\max}(\Delta_{\wedge,\textup{Dir}}) =
\Dom_{\wedge,\min}(\Delta_{\wedge,\textup{Dir}}) \oplus
\Span\{\omega(x) \varphi(\theta)x^{\pi/\alpha},\omega(x)\varphi(\theta)x^{-\pi/\alpha}\}
$$
as above. This induces an isomorphism
$$
\Dom_{\wedge,\max}/\Dom_{\wedge,\min} \cong 
\Span\{\varphi(\theta)x^{\pi/\alpha},\varphi(\theta)x^{-\pi/\alpha}\},
$$
and the scaling action $\kappa_{\varrho}$ on $\Dom_{\wedge,\max}/\Dom_{\wedge,\min}$
takes the form
$$
\kappa_{\varrho}\bigl(\varphi(\theta)x^{\pi/\alpha}\bigr) =
\varrho^{\pi/\alpha}\cdot\varphi(\theta)x^{\pi/\alpha}
\textup{ and }
\kappa_{\varrho}\bigl(\varphi(\theta)x^{-\pi/\alpha}\bigr) =
\varrho^{-\pi/\alpha}\cdot\varphi(\theta)x^{-\pi/\alpha}
$$
on the basis elements in the image of this isomorphism. Choose an arbitrary domain
$\Dom_{\wedge} \subset \Dom_{\wedge,\max}$ with
$\dim\Dom_{\wedge}/\Dom_{\wedge,\min} = 1$. $\Dom_{\wedge}$ is represented by
$\Span\{a\cdot\varphi(\theta)x^{\pi/\alpha} + b\cdot\varphi(\theta)x^{-\pi/\alpha}\}$
for some $(a,b) \neq (0,0)$. In the Grassmannian of $1$-dimensional subspaces of
$\Dom_{\wedge,\max}/\Dom_{\wedge,\min}$ we get
\begin{gather*}
\kappa_{\varrho}\bigl(\Dom_{\wedge}/\Dom_{\wedge,\min}\bigr)
= \Span\{\varrho^{\pi/\alpha}a\cdot\varphi(\theta)x^{\pi/\alpha} +
\varrho^{-\pi/\alpha}b\cdot\varphi(\theta)x^{-\pi/\alpha}\} \\
= \Span\{\varrho^{2\pi/\alpha}a\cdot\varphi(\theta)x^{\pi/\alpha} +
b\cdot\varphi(\theta)x^{-\pi/\alpha}\}
\underset{\varrho \to 0}{\longrightarrow}
\begin{cases}
\Span\{\varphi(\theta)x^{\pi/\alpha}\} & \textup{ if } b = 0, \\
\Span\{\varphi(\theta)x^{-\pi/\alpha}\} & \textup{ if } b \neq 0.
\end{cases}
\end{gather*}
It is easy to see that both domains $\Dom_{\wedge,\pm\alpha} =
\Dom_{\wedge,\min}(\Delta_{\wedge,\textup{Dir}})
\oplus\Span\{\omega(x)\varphi(\theta)x^{\pm\pi/\alpha}\}$ are selfadjoint for
$\Delta_{\wedge,\textup{Dir}}$. Consequently, every ray $\Gamma \subset \C$
not parallel to the real line is a ray of minimal growth for
$\Delta_{\wedge,\textup{Dir}}$ for all domains $\Dom_{\wedge} \subset
\Dom_{\wedge,\max}$ with $\dim\Dom_{\wedge}/\Dom_{\wedge,\min} = 1$.
The reasoning here is completely analogous to Example~\ref{Example}.

Theorem~\ref{MainTheoremBVP} now applies, and we conclude that the system of
generalized eigenfunctions of $\Delta_{\textup{Dir}}$ is complete
in $x^{-1}L^2_b(\Mbar)$ for all domains $\Dom \subset \Dom_{\max}(\Delta_{\textup{Dir}})$
with $\dim\Dom/\Dom_{\min}(\Delta_{\textup{Dir}}) = 1$ as was claimed.
\end{example}

%%%%%%%%%%%%%%%%%%%%%%%%%%%%%%%%%%%%%%%%%%%%%%%%%%%%%%%%%%%%%%%%%%%%%


\begin{thebibliography}{10}
%%%%%%%%%%%%%%%%%%%%%%%%%%%%%%%%%%%%%%%%%%%%%%%%%%%%%%%%%%%%%%%%%%%%%

\bibitem{Agmon}
S.~Agmon, \emph{On the eigenfunctions and on the eigenvalues of general elliptic
boundary value problems}, Comm. Pure Appl. Math. \textbf{15} (1962), 119--147.

\bibitem{Agranovich1}
M.S.~Agranovich, \emph{Elliptic operators on closed manifolds}, in:
\emph{Partial Differential Equations~VI} (Yu.V.~Egorov and M.A.~Shubin (eds.)),
Encyclopaedia of Mathematical Sciences, vol.~63, pp.~1--130, Springer-Verlag,
Berlin-Heidelberg, 1994.

\bibitem{Agranovich2}
\bysame, \emph{Elliptic boundary problems}, in:
\emph{Partial Differential Equations~IX} (M.S.~Agranovich, Yu.V.~Egorov,
and M.A.~Shubin (eds.)),
Encyclopaedia of Mathematical Sciences, vol.~79, pp.~1--144, Springer-Verlag,
Berlin-Heidelberg, 1997.

\bibitem{BuzanoToft}
E.~Buzano and J.~Toft, \emph{Schatten-von Neumann properties in the Weyl calculus},
Preprint 2008 (Preprint arXiv:0809.1207 on arXiv.org), 27 pages.

\bibitem{CorSchroSei}
S.~Coriasco, E.~Schrohe, and J.~Seiler, \emph{Realizations of differential operators
on conic manifolds with boundary}, Ann.~Global Anal.~Geom \textbf{31} (2007), no.~3,
223--285.

\bibitem{DunfordSchwartz}
N.~Dunford and J.T.~Schwartz, \emph{Linear Operators. Part II: Spectral Theory.
Selfadjoint Operators in Hilbert Space}, John Wiley \& Sons, New York, 1963.

\bibitem{EgorovKondratievSchulze1}
Yu.~Egorov, V.~Kondratiev, and B.-W.~Schulze, \emph{Completeness of eigenfunctions
of an elliptic operator on a manifold with conical points},
Russ. J. Math. Phys. \textbf{8} (2001), no. 3, 267--274.
See also C. R. Acad. Sci. Paris S{\'e}r. I Math. \textbf{333} (2001), no. 6, 551--556.

\bibitem{EgorovKondratievSchulze2}
\bysame, \emph{On completeness of root functions of elliptic boundary problems in
a domain with conical points on the boundary},
C. R. Math. Acad. Sci. Paris \textbf{334} (2002), no. 8, 649--654.

\bibitem{GilHeatTrace}
J.~Gil, \emph{Full asymptotic expansion of the heat trace for non-self-adjoint
elliptic cone operators}, Math. Nachr. \textbf{250} (2003), 25--57.

\bibitem{GKM1}
J.~Gil, T.~Krainer, and G.~Mendoza, \emph{Geometry and spectra of 
closed extensions of elliptic cone operators}, 
Canad. J. Math. \textbf{59} (2007), no. 4, 742--794.

\bibitem{GKM2}
\bysame, \emph{Resolvents of elliptic cone operators}, J. Funct. Anal. \textbf{241}
(2006), no. 1, 1--55.

\bibitem{GKM3}
\bysame, \emph{On rays of minimal growth for elliptic cone operators}, Oper.~Theory Adv.~Appl. \textbf{172} (2007), 33--50.

\bibitem{GKM4}
\bysame, \emph{Trace expansions for elliptic cone operators with stationary domains},
Preprint 2008 (math.SP/0811.3776 on arXiv.org). To appear in Trans.~Amer.~Math.~Soc.

\bibitem{GKM5}
\bysame, \emph{Dynamics on Grassmannians and resolvents of cone operators},
Preprint 2009 (math.AP/0907.0023 on arXiv.org).

\bibitem{GiMe01}
J.~Gil and G.~Mendoza, \emph{Adjoints of elliptic cone operators},
Amer. J. Math. \textbf{125} (2003), no. 2, 357--408.

\bibitem{KrainBVPCone}
T.~Krainer, \emph{Resolvents of elliptic boundary problems on conic manifolds},
Comm.~Partial Differential Equations \textbf{32} (2007), 257--315.

\bibitem{Le97} 
M.~Lesch, \emph{Operators of {F}uchs type, conical singularities, and asymptotic
methods}, Teubner-Texte zur Math. vol 136, B.G. Teubner, Stuttgart, Leipzig, 1997.

\bibitem{LoRes01}
P.~Loya, \emph{On the resolvent of differential operators on conic manifolds}, 
Comm. Anal. Geom. \textbf{10} (2002), no. 5, 877--934.

\bibitem{RBM2} 
R.~Melrose, \emph{The Atiyah-Patodi-Singer index theorem}, Research Notes in Mathematics, A~K~Peters, Ltd., Wellesley, MA, 1993.

\bibitem{SchulzeNH} 
B.--W. Schulze, \emph{Pseudo-differential operators on manifolds with singularities}, 
Studies in Mathematics and its Applications, 24. North-Holland Publishing Co., 
Amsterdam, 1991.

\bibitem{Toft}
J.~Toft, \emph{Schatten-von Neumann properties in the Weyl calculus, and calculus
of metrics on symplectic vector spaces}, Ann. Global Anal. Geom. \textbf{30} (2006),
no. 2, 169--209.

\end{thebibliography}
\end{document}